\newif\iffiif
\fiiffalse 

\iffiif
\documentclass[]{fiif}
\else
\documentclass[b5paper]{amsart}
\fi

\usepackage[utf8]{inputenc}
\usepackage{lmodern}
\usepackage[T1]{fontenc}

\iffiif
\else
\usepackage[twoside,hmargin={0.575in,0.575in},vmargin={0.4in,0.5in},includehead]{geometry} 
\usepackage{amsthm} 
\def\theenumi{\roman{enumi}} 
\fi

\usepackage{amsmath,amssymb,mathtools,bm,textcomp,stmaryrd} 
\usepackage{booktabs,array} 
\usepackage[shrink=50,stretch=20]{microtype} 
\usepackage{graphicx,xcolor} 

\newif\ifTIKZ
\ifTIKZ
  \usepackage{tikz} 
  \usetikzlibrary{arrows, arrows.meta, cd, hobby, positioning, automata}
  \usetikzlibrary{external}
  \newcommand\TIKZinclude[1]{\tikzsetnextfilename{#1}\input{#1.tex}}
\else
  \makeatletter
  \def\Xincludegraphics#1{\begingroup\leavevmode\Xpgfexternalreaddpth{#1}\setbox1=\hbox{\includegraphics{#1}}%
    \ifdim\Xpgfretval=0pt \box1 \else\dimen0=\Xpgfretval\relax\hbox{\lower\dimen0 \box1 }\fi\endgroup}
  \newread\Xr@pgf@reada
  \def\Xpgfexternalreaddpth#1{\edef\Xpgfexternalreaddpth@restore{\noexpand\endlinechar=\the\endlinechar\space
    \noexpand\catcode`\noexpand\@=\the\catcode`\@\space}\def\Xpgfretval{0pt}\endlinechar=-1 \catcode`\@=11 %
    \openin\Xr@pgf@reada=#1.dpth \Xpgfincludeexternalgraphics@read@dpth\Xpgfexternalreaddpth@restore}
  \def\Xpgfincludeexternalgraphics@read@dpth{\ifeof\Xr@pgf@reada\closein\Xr@pgf@reada\else\read\Xr@pgf@reada
    to\Xpgfincludeexternalgraphics@auxline\ifx\Xpgfincludeexternalgraphics@auxline\empty\else
    \expandafter\Xpgfincludeexternalgraphics@read@dpth@line\Xpgfincludeexternalgraphics@auxline
    \Xpgfincludeexternalgraphics@read@dpth@line@EOI\fi\expandafter\Xpgfincludeexternalgraphics@read@dpth\fi}
  \long\def\Xpgfincludeexternalgraphics@read@dpth@line#1#2\Xpgfincludeexternalgraphics@read@dpth@line@EOI{%
    \ifcat\noexpand#1\relax\if@filesw{\toks0={#1#2}\immediate\write\@auxout{\noexpand\def\noexpand\Xdpthimport{%
    \the\toks0 }\noexpand\Xdpthimport }}\fi\else\def\Xpgfretval{#1#2}\fi}%
  \makeatother
  \newcommand\TIKZinclude[1]{\Xincludegraphics{#1}}
\fi

\DeclarePairedDelimiterX\set[2]\lbrace\rbrace{\,#1\suchthat#2\,} 
\DeclarePairedDelimiter\qfl\lfloor\rfloor 
\DeclarePairedDelimiter\abs\lvert\rvert 
\DeclareMathOperator\omod{mod} 
\newcommand*\f{\mathrm{f}} 
\newcommand*\N{\mathbb{N}} 
\newcommand*\R{\mathbb{R}} 
\newcommand*\Q{\mathbb{Q}} 
\newcommand*\Z{\mathbb{Z}} 
\newcommand*\K{\mathbb{K}} 
\newcommand*\Kf{K_\f} 
\newcommand*\frakp{\mathfrak{p}} 
\newcommand*\RR{\mathcal{R}} 
\newcommand*\QQ{\mathcal{Q}} 
\newcommand*\OO{\mathcal{O}} 
\newcommand*\XX{\mathcal{X}} 
\newcommand*\TT{\mathcal{T}} 
\newcommand*\BB{\mathcal{B}} 
\AtBeginDocument{
  \let\h\relax 
  \let\AA\relax 
  \newcommand*\h{{\bm{h}}} 
  \newcommand*\AA{\mathcal{A}} 
}
\newcommand*\defined[1]{\emph{#1}} 
\newcommand*\suchthat{:} 
\newcommand*\inverselim{\varprojlim} 
\newcommand*\eqdef{\coloneqq} 
\newcommand*\OL[1]{\overline{#1}} 
\let\ol\relax
\newcommand*\divides{\mathrel{|}} 
\newcommand*\ndivides{\mathrel{\nmid}} 
\newcommand*\s{\raise0.2ex\hbox{\scriptsize$\star$}} 

\newcommand*\tight[1]{{\mkern1mu\relax#1\mkern1mu}}

\iffiif
\def\Alphenumi{\def\theenumi{(\Alph{enumi})}} 
\else
\def\Alphenumi{\def\theenumi{\Alph{enumi}}} 
\fi

\def\xPref[#1]#2{#2\begingroup\llbracket\thinmuskip10mu\relax#1\rrbracket\endgroup}
\DeclareRobustCommand\Pref{\xPref}

\makeatletter
\def\per@skip{1mu}
\def\perX*#1{{#1}^\omega}
\def\perY#1{{(#1)}^\omega}
\DeclareRobustCommand\per{\@ifnextchar*\perX\perY}
\makeatother

\makeatletter
\usepackage{etoolbox}
\iffiif 
  \let\qedhere\relax
\else 
  \let\xproof\proof
  
  \def\yyproof[#1]{\xproof[Proof #1]}
  \def\yproof{\@ifnextchar[\yyproof\xproof}
  \let\proof\yproof
  
\fi 
\makeatother

\numberwithin{equation}{section}

\makeatletter
\def\th@plain{\slshape}
\makeatother
\newtheorem{lemma}{Lemma}[section]
\newtheorem{proposition}[lemma]{Proposition}

\newtheorem{theorem}{Theorem} 
\newtheorem{Theorem}{Theorem} 

\iffiif
\let\newtheorem\newnumbered 
\else
\theoremstyle{definition} 
\fi
\newtheorem{definition}[lemma]{Definition}
\newtheorem{example}[lemma]{Example}

\newtheorem{remark}[lemma]{Remark}

\usepackage{hyperref} 
\hypersetup{
    colorlinks,
    linkcolor={red!50!black},
    citecolor={green!50!black},
    urlcolor={blue!90!black}
}

\title[Beta-expansions of rational numbers]{Beta-expansions of rational numbers in~quadratic~Pisot~bases}

\iffiif 
\author{Tom\'a\v s Hejda and Wolfgang Steiner}
\classno{11A63 (11R06 37B10)}
\extraline{%
The first author acknowledges support by the Grant Agency of the Czech Technical University in Prague grant SGS11/162/OHK4/3T/14
 and the Czech Science Foundation grant 13-03538S.
The second author acknowledges support by the ANR/FWF project ``FAN -- Fractals and Numeration'' (ANR-12-IS01-0002, FWF grant I1136)
 and by the ANR project ``DYNA3S'' (ANR-13-BS02-0003).}
\else 
\author[T. Hejda]{Tom\'a\v s Hejda} 
\address{Dept.\@ Math.\@ FNSPE, Czech Technical University in Prague, Trojanova~13, Prague 12000, Czech Rep.
\newline\hspace*{\parindent}%
LIAFA, CNRS UMR 7089, Universit\'e Paris Diderot -- Paris 7, Case 7014, 75205 Paris Cedex~13, France} 
\email{tohecz@gmail.com} 
\author[W. Steiner]{Wolfgang Steiner} 
\address{LIAFA, CNRS UMR 7089, Universit\'e Paris Diderot -- Paris 7, Case 7014, 75205 Paris Cedex~13, France} 
\email{steiner@liafa.univ-paris-diderot.fr} 
\fi 

\begin{document}

\nonfrenchspacing

\def\bar#1{{\color{red}#1}}

\begin{abstract}
We study rational numbers with purely periodic Rényi $\beta$-expansions.
For bases $\beta$ satisfying $\beta^2=a\beta+b$ with $b$ dividing $a$, we give a~necessary and sufficient condition for $\gamma(\beta)=1$,
 i.e., that all rational numbers $p/q\in[0,1)$ with $\gcd(q,b)=1$ have a~purely periodic $\beta$-expansion.
A simple algorithm for determining the value of $\gamma(\beta)$ for all quadratic Pisot numbers $\beta$
 is described.
\end{abstract}

\iffiif\else
\maketitle 
\fi

\section{Introduction and main results}

Rényi $\beta$-expansions~\cite{renyi_1957} provide a~very natural generalization of standard positional
 numeration systems such as the decimal system.
Let $\beta>1$ denote the base.
Expansions of numbers $x\in[0,1)$ are defined in terms of the $\beta$-transformation
\[
	T\colon[0,1)\to[0,1),\; x\mapsto \beta x-\qfl{\beta x}
.\]
The expansion of $x$ is the infinite string $x_1 x_2 x_3\cdots$ where $x_j\eqdef \qfl{\beta T^{j-1}x}$.
For $\beta\in\N$, we recover the standard expansions in base $\beta$
 and the $\beta$-expansion of $x\in[0,1)$ is eventually periodic
 (i.e., there exist $p,n$ such that $x_{k+p}=x_k$ for all $k\geq n$) if and only if $x\in\Q$.
This result was generalized to all Pisot bases by Schmidt~\cite{schmidt_1980},
 who proved that for a~Pisot number $\beta$ the expansion of $x\in[0,1)$ is eventually periodic
 if and only if $x$ is an element of the number field $\Q(\beta)$.
Moreover, he showed that when $\beta$ satisfies $\beta^2=a\beta+1$, then each $x\in[0,1)\cap\Q$
 has a~purely periodic $\beta$-expansion.

Akiyama~\cite{akiyama_1998} showed that if $\beta$ is a~Pisot unit satisfying a~certain finiteness property then
 there exists $c>0$ such that all rational numbers $x\in\Q\cap[0,c)$ have a~purely periodic expansion.
If $\beta$ is not a~unit, then a~rational number $p/q\in[0,1)$ can have a~purely periodic expansion
 only if $q$ is co-prime to the norm $N(\beta)$.
Many Pisot non-units satisfy that there exists $c>0$ such that all rational numbers $\frac pq\in[0,c)$
 with $q$ co-prime to $b$ have a purely periodic expansion.
This stimulates for the following definition:

\begin{definition}\label{def:gamma}
Let $\beta$ be a~Pisot number, and let $N(\beta)$ denote the norm of~$\beta$.
Then we define $\gamma(\beta)\in[0,1]$ as the maximal $c$ such that all $\frac pq\in\Q\cap[0,c)$
 with $\gcd(q,N(\beta))=1$ have a~purely periodic $\beta$-expansion.
In other words,
\[
	\gamma(\beta)\eqdef \inf\set[\big]{\tfrac pq\in\Q\cap[0,1)}{\gcd(q,b)=1,\, \tfrac pq \text{ has a not purely periodic expansion}}\cup\{1\},
\]
 where $b=\abs{N(\beta)}$ is the norm of $\beta$.
\end{definition}

The question is how to determine the value of $\gamma(\beta)$.
As well, knowing when $\gamma(\beta)=0$ or $1$ is of big interest.
Values of $\gamma(\beta)$ for whole classes of numbers as well as for particular numbers
 have been given \cite{akiyama_1998,ABBS_2008,akiyama_scheicher_2005,minervino_steiner_2014,schmidt_1980}.

It is easy to observe that the expansion of $x$ is purely periodic if and only if $x$ is a~periodic point of $T$, i.e.,
 there exists $p\geq1$ such that $T^px=x$.
The natural extension $(\XX,\TT)$ of the dynamical system $([0,1),T)$ (w.r.t.\@ its unique absolutely continuous invariant measure)
 can be defined in an algebraic way, cf.~\S\ref{subsect:tiles}.
Several authors contributed to proving the following result:
A~point $x\in[0,1)$ has a purely periodic $\beta$-expansion if and only if $x\in\Q(\beta)$ and its diagonal embedding
 lies in the natural extension domain $\XX$.
The quadratic unit case was solved by Hama and Imahashi~\cite{hama_imahashi_1997},
 the confluent unit case by Ito and Sano~\cite{ito_sano_2001,ito_sano_2002}.
Then Ito and Rao~\cite{ito_rao_2005} resolved the unit case completely using an algebraic argument.
For non-unit bases~$\beta$, one has to consider finite ($p$-adic) places of the field $\Q(\beta)$.
This consideration allowed Berthé and Siegel~\cite{berthe_siegel_2007} to expand the result to all (non-unit) Pisot numbers.

The first values of $\gamma(\beta)$ for two particular quadratic non-units were provided by Akiyama et al.~\cite{ABBS_2008}.
Recently, Minervino and Steiner~\cite{minervino_steiner_2014} described the boundary of $\XX$ for quadratic non-unit Pisot bases.
This allowed them to find the value of $\gamma(\beta)$ for an infinite class of quadratic numbers:

\begin{Theorem}[{\cite{minervino_steiner_2014}}]\label{thm:A}
Let $\beta$ be the positive root of $\beta^2=a\beta+b$ for $a\geq b\geq1$ two co-prime integers.
Then
\[
	\gamma(\beta)=\begin{cases}
		1-\frac{(b-1)b\beta}{\beta^2-b^2} & \text{if $a>b(b-1)$,}
	\\
		0 & \text{otherwise.}
	\end{cases}
\]
In particular, $\gamma(\beta)=1$ if and only if $b=1$.
\end{Theorem}

The purpose of this article is to generalize Theorem~\ref{thm:A} to all quadratic Pisot numbers with norm $N(\beta)<0$.
(Note that when $N(\beta)>0$, then $\beta$ has a~positive
 Galois conjugate $\beta'>0$ and $\gamma(\beta)=0$ by \cite[Proposition~5]{akiyama_1998}.)
To this end, we define $\beta$-adic expansions (not to be confused with the Rényi $\beta$-expansions)
 similarly to $p$-adic expansions with $p\in\Z$, see also~\S\ref{sect:Hensel}.

\begin{definition}\label{def:H}
Let $\beta$ be an algebraic integer.
The \defined{$\beta$-adic expansion} of $x\in\Z[\beta]$ is the unique infinite word $\h(x)\eqdef u_0u_1u_2\dotsm$
 such that $u_n\in\{0,1,\dots,\abs{N(\beta)}-1\}$  and
\(
	x-\sum_{i=0}^{n-1} u_i\beta^i\in\beta^n\Z[\beta]
\) for all $n\in\N$.
\end{definition}
 
\begin{theorem}\label{thm:gamma-bar}
Let $\beta$ be a~quadratic Pisot number, root of $\beta^2=a\beta+b$ with $a\geq b\geq1$.
Then
\[
	\gamma(\beta) = \begin{cases}
		0
		&\quad
		\text{if $\sup_{j\in\Z} P_{\h(j-\beta)}(\beta')>\beta$ or $\inf_{j\in\Z} P_{\h(j)}(\beta')<-1$,}
	\\[1ex]
		\beta-a
		&\quad
		\text{if $\sup_{j\in\Z} P_{\h(j-\beta)}(\beta')\in(2\beta-a-1,\beta]$ }
		\text{and $\inf_{j\in\Z} P_{\h(j)}(\beta')\geq\beta-a-1$,}
	\\[1ex]
		\mathrlap{
		1+\inf_{j\in\Z} P_{\h(j)}(\beta')
		\qquad
		\text{otherwise,}
		}
	\end{cases}
\]
 where $P_{u_0u_1u_2\dotsm}(X)\eqdef \sum_{n\geq0} u_nX^n$.
\end{theorem}

In many cases, we obtain the following direct formula (which we conjecture to be true for all $a\geq b\geq1$):

\begin{theorem}\label{thm:gamma0}
Let $\beta$ be a~quadratic Pisot number, root of $\beta^2=a\beta+b$ for $a\geq b\geq1$.
Suppose $a>\frac{1+\sqrt5}{2}b$ or $a=b$ or $a\perp b$.
Then
\begin{align}\label{eq:gamma0}
	\gamma(\beta)&=\max\Bigl\{ 0, 1 + \inf_{j\in\Z} P_{\h(j)}(\beta') \Bigr\}
.\end{align}
\end{theorem}

The infimum in \eqref{eq:gamma0} can be computed easily with the help of Proposition~\ref{prop:estim} below.
In the case $\frac ab\in\Z$, Proposition~\ref{prop:estim-2} provides an even faster algorithm,
 and we are able to prove a~necessary and sufficient condition for $\gamma(\beta)=1$:

\begin{theorem}\label{thm:abZ}
Let $\beta$ be a~quadratic Pisot number, root of $\beta^2=a\beta+b$ with $a\geq b\geq1$ and such that $b$ divides $a$.
\begin{enumerate}

\item
We have that $\gamma(\beta)=1$ if and only if $a\geq b^2$ or $(a,b)\in\{(24,6),(30,6)\}$.

\item
If $a=b\geq3$ then $\gamma(\beta)=0$.

\end{enumerate}
\end{theorem}

This paper is organized as follows:
In the next section, notions on words, representation spaces and $\beta$-tiles are recalled,
 and properties of $\beta$-adic expansions are studied.
Section~\ref{sect:tiles} connects tiles arising from the $\beta$-transformation and the value $\gamma(\beta)$
 in order to prove Theorem~\ref{thm:gamma-bar}.
The proof of Theorem~\ref{thm:gamma0} is completed in Section~\ref{sect:b-div-a}, together with that of Theorem~\ref{thm:abZ}.
Comments on the general case are in Section~\ref{sect:general}, along with a list of related open questions.

\section{Preliminaries}

\subsection{Words over a finite alphabet}
We consider both finite and infinite words over a~finite alphabet~$\AA$.
The set of finite words over $\AA$ is denoted $\AA^*$.
The set of all (right) infinite words over $\AA$ is denoted $\AA^\omega$, and it is equipped with the Cantor topology.
An infinite word is \defined{(eventually) periodic} if it is of the form $v\per*{u}\eqdef vuuu\dotsm$;
 a~finite word $v$ is the pre-period and a non-empty finite word $u$ is the period;
 if the pre-period is empty, we speak about a~\defined{purely periodic word}.
A prefix of a~(finite or infinite) word~$w$ is any finite word $v$ such that $w$ can be written as
 $w=vu$ for some word $u$.
We denote by $\Pref[n]{\bm u}$ the prefix of length $n$ of an infinite word $\bm u$.

To a~finite word $w=w_0w_1\dots w_{k-1}$ we assign the polynomial
\[
	P_w(X) \eqdef \sum_{i=0}^{k-1} w_iX^i
.\]
Similarly, $P_{\bm u}(X) \eqdef \sum_{i\geq0} u_iX^i$ is a power series for an infinite word $\bm u=u_0u_1u_2\dotsm$.

\subsection{Representation spaces}

The following notation will be used:
For integers $a,b\in\Z$, we denote by $a\perp b$ the fact that
 $a$ and $b$ are co-prime, i.e., that $\gcd(a,b)=1$.
Moreover, for $b\geq2$ we put $\Z_b\eqdef\set{p/q}{p,q\in\Z,\, q\perp b}$
 (the ring of rational numbers with denominator co-prime to $b$).

We adopt the notation of \cite{minervino_steiner_2014}, however, we restrict ourselves to $\beta$ being
 a~quadratic Pisot number.
Let $K=\Q(\beta)$.
Since $\beta$ is quadratic, there are exactly two infinite places of $K$;
 they are given by the two Galois isomorphisms of $\Q(\beta)$: the identity and $x\mapsto x'$
 that maps $\beta$ to its Galois conjugate.
Both these places have $\R$ as their completion.

If $\beta$ is not a~unit, then we have to consider finite places of $K$ as well.
We define the ring $\Kf$ as the direct product $\Kf\eqdef \prod_{\frakp\divides (\beta)} K_\frakp$,
 where $\frakp$ runs through all prime ideals of $\Q(\beta)$ that divide the principal ideal $(\beta)$
 and $K_\frakp$ is the associate completion of $\K$;
 for a precise definition, we refer to~\cite[\S2.2]{minervino_steiner_2014}.
The direct products $\K\eqdef K\times K'\times\Kf$ and $\K'\eqdef K'\times\Kf$ are called \defined{representation spaces}.
We define the diagonal embeddings
\[
	\delta\colon\Q(\beta)\to\K, \; x\mapsto(x,x',x_\f)
\quad\text{and}\quad
	\delta'\colon\Q(\beta)\to\K', \; x\mapsto(x',x_\f)
,\]
 where $x_\f$ is the vector of the embeddings of $x$ into the spaces $K_\frakp$.
We put
\[
	S_\f \eqdef \OL{\set{x_\f}{x\in S}}
\quad
	\text{for any $S\subseteq K$.}
\]
In particular, we consider $\ol{\Z[\beta]_\f}$, which is a compact subset of $\Kf$.
Since multiplication by $\beta_\f$ is a contraction on $K_\f$, we have that $\beta_\f^n\Z[\beta]_\f \to \{0_\f\}$ as $n\to\infty$.


\subsection{Beta-tiles}\label{subsect:tiles}

For $x\in[0,1)$, we define the (reflected and translated) \defined{$\beta$-tile} of $x$ as the Hausdorff limit
\[
	\QQ(x)\eqdef \lim_{k\to\infty} \delta'\bigl(x-\beta^k T^{-k}(x)\bigr) \subseteq \K'
.\]
Note that the standard definition of a~$\beta$-tile for $x\in\Z[\beta^{-1}]\cap[0,1)$ is $\RR(x)\eqdef\delta'(x)-\QQ(x)$,
 see e.g.~\cite{minervino_steiner_2014}.
For a quadratic Pisot number $\beta$, root of $\beta^2=a\beta+b$ with $a\geq b\geq1$,
 we have that $\QQ(x)=\QQ(0)$ for $x<\beta-a$ and $\QQ(x)=\QQ(\beta-a)$ otherwise.
The dynamical system $([0,1),T)$ admits $(\XX,\TT)$ as its natural extension,
 where
\begin{gather*}
	\XX \eqdef
		\bigl( [0,\beta-a)\times\QQ(0) \bigr)
	\cup
		\bigl( [\beta-a,1)\times\QQ(\beta-a) \bigr)
	\subset\K
\end{gather*}
 is a union of two suspensions of $\beta$-tiles
 and $\TT(x,y)\eqdef \delta(\beta)(x,y)-\delta(\qfl{\beta x})$.
The natural extension domain is often required to be a~closed set,
 but here it is more convenient to work with the one above, since the following result holds:
 
\begin{Theorem}[\cite{hama_imahashi_1997,ito_rao_2005,berthe_siegel_2007}]\label{thm:Pur-XX}
For a~Pisot number $\beta$, we have that $x$ has a~purely periodic $\beta$-expansion if and only if $x\in\Q(\beta)$
 and $\delta(x)\in\XX$.
\end{Theorem}

\subsection{Beta-adic expansions}\label{sect:Hensel}

In Definition~\ref{def:H}, $\beta$-adic expansions are defined on $\Z[\beta]$.
By Lemma~\ref{lem:H} below, we extend this definition to the closure $\Z[\beta]_\f$
 similarly to the $p$-adic case.
To this end, let
\[
	D\colon \ol{\Z[\beta]_\f} \to \ol{\Z[\beta]_\f}
,\quad x \mapsto \beta_\f^{-1} \bigl(z-d(z)_\f\bigr)
,\]
 where $d(x)$ is the unique digit $d\in\AA\eqdef\{0,1,\dots,\abs{N(\beta)}-1\}$ such that $\beta_\f^{-1}(x-d_\f)$ is in $\ol{\Z[\beta]_\f}$.
Such $d$ exists because $\Z[\beta] = \AA + \beta \Z[\beta]$.
It is unique because $(c+\ol{\beta \Z[\beta]})_\f \cap (d+\ol{\beta \Z[\beta]})_\f \neq \emptyset$
 implies $(\beta^{-1} (c-d))_\f \in \ol{\Z[\beta]_\f}$ and thus $c \equiv d\pmod{N(\beta)}$ by the following lemma:

\begin{lemma}[{\cite[Lemma~5.2 and Eq.~(5.1)]{minervino_steiner_2014}}]\label{lem:5.2}
For each $x\in\Z[\beta^{-1}]\setminus\Z[\beta]$ we have $x_\f\notin\ol{\Z[\beta]_\f}$.
There exists $k\in\N$ such that $\Z[\beta^{-1}]\cap \beta^k\OO\subseteq\Z[\beta]$,
 where $\OO$ is the ring of integers in $\Q(\beta)$.
\end{lemma}

\begin{lemma}\label{lem:H}
The $\beta$-adic expansion map $\h_\f:\Z[\beta]_\f\to\AA^\omega$ defined by
\[
	\h_\f(z)\eqdef u_0u_1u_2\dotsm
,\quad\text{where}\quad
	u_i\eqdef d\bigl(D^i(z)\bigr)
,\]
 is a homeomorphism.
It satisfies that $\h_\f(x_\f)=\h(x)$ for all $x\in\Z[\beta]$.
\end{lemma}

\begin{proof}
The map $\h_\f$ is surjective because $\h_\f(P_{\bm u}(\beta_\f)) = \bm u$ for all $\bm u\in\AA^\omega$.
It is injective because $\h_\f(z) = \bm u = u_0u_1u_2\dotsm$ implies that
 $z \in \sum_{i=0}^{n-1} u_i \beta_\f^i + \ol{\beta_\f^n \Z[\beta]_\f}$ for all $n$, thus $z = P_{\bm u}(\beta_\f)$.
 
Since $\OO_\f$ is open and $\Z[\beta^{-1}]_\f=K_\f$,
 we get from Lemma~\ref{lem:5.2} that $\Z[\beta]_\f = \bigcup_{x\in\Z[\beta]} x_\f+\beta_\f^k\OO_\f$ for some $k\in\N$,
 and therefore it is an open set as well.
Then the preimage $\h_\f^{-1}(v\AA^\omega) = P_{v}(\beta_\f)+\beta_\f^n\Z[\beta]_\f$ is open for any $v\in\AA^*$.
As the cylinders $\set{v\AA^\omega}{v\in\AA^*}$ form a base of the topology of $\AA^\omega$,
 the map $\h_\f$ is continuous.

The inverse $\h_\f^{-1}$ is continuous because $\beta_\f^n\Z[\beta]_\f\to\{0_\f\}$ as $n\to\infty$.

For $x\in\Z[\beta]$, the equality $\h_\f(x_\f)=\h(x)$ follows from the fact that $\beta^{-1}(x-d(x_\f))\in\Z[\beta]$.
\end{proof}

Note that we can also identify $\Z[\beta]_\f$ with the inverse limit space $\inverselim \Z[\beta]/\beta^n\Z[\beta]$.
Indeed defining the isomorphism
\[
	\kappa\colon\AA^\omega\to\inverselim \Z[\beta]/\beta^n\Z[\beta],
\quad
	u_0u_1u_2\dotsm\mapsto (\xi_1,\xi_2,\xi_3,\dotsc)
,\quad\text{where}\quad
	\xi_n = \sum_{i=0}^{n-1} u_i\beta^i
,\]
 the following diagram commutes:
\[
	\TIKZinclude{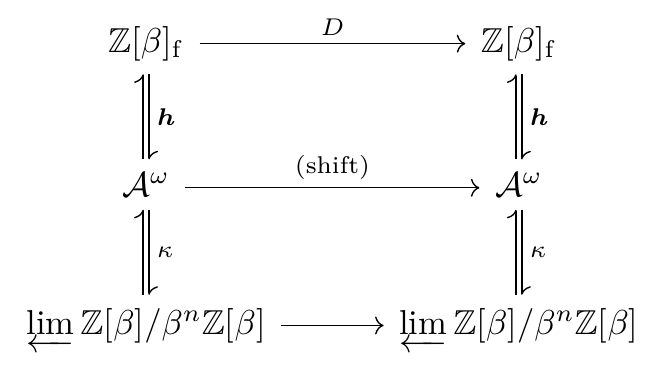}
\]

\section{\texorpdfstring{Beta-tiles and the value $\gamma(\beta)$}{Beta-tiles and the value gamma(beta)}}\label{sect:tiles}

The goal of this section is to prove Theorems~\ref{thm:gamma-bar} and~\ref{thm:gamma0},
 using the connection between $\beta$-tiles and the value of $\gamma(\beta)$.
First we prove the following lemma about the closures of $\Z$ and $\Z_b$ in $K_\f$:

\begin{lemma}\label{lem:closure}
We have that $\ol{(\Z)_\f}=\ol{(\Z_b)_\f}=\ol{(\Z_b\cap[c,d])_\f}$ for all $c<d$.
\end{lemma}

\begin{proof}
We have that $\ol{(\Z_b)_\f}=\ol{(\Z_b\cap[c,d])_\f}$ by \cite[Lemma 4.7]{ABBS_2008}.
Clearly $\Z\subseteq\Z_b$ whence $(\Z)_\f\subseteq(\Z_b)_\f$.
We will prove that $(\Z_b)_\f\subseteq\ol{(\Z)_\f}$,
 namely that every point $x/q\in\Z_b$ for $x,q\in\Z$ and $q\perp b$ can be approximated by integers.
For each $n\in\N$, there exists $q_n\in\Z$ such that $q_n q\equiv1\pmod{b^n}$.
Then
\(
	\frac{x}{q}-q_nx
	= (1-q_nq)\frac{x}{q}
	\in \frac1q b^n\Z \subseteq \frac1q \beta^n\Z[\beta]
,\)
 therefore $(q_nx)_\f\to(x/q)_\f$.
\end{proof}

\begin{proof}[of Theorem~\ref{thm:gamma-bar}]
By Definition~\ref{def:gamma}, Theorem~\ref{thm:Pur-XX} and since $\delta(1)\notin\XX$,
 we have that
\[
	\gamma(\beta)=\inf\set[\big]{x\in\Z_b}{x\geq0 ,\, \delta(x)\notin\XX}
.\]
For $x\in\Q\cap[0,\beta-a)$, the condition $\delta(x)\in\XX$ is equivalent to $\delta'(x)\in\QQ(0)$;
 for $x\in\Q\cap[\beta-a,1)$, it is equivalent to $\delta'(x)\in\QQ(\beta-a)$.

\begin{figure}
\centering
\iffiif\else\footnotesize\fi
\oalign{
\TIKZinclude{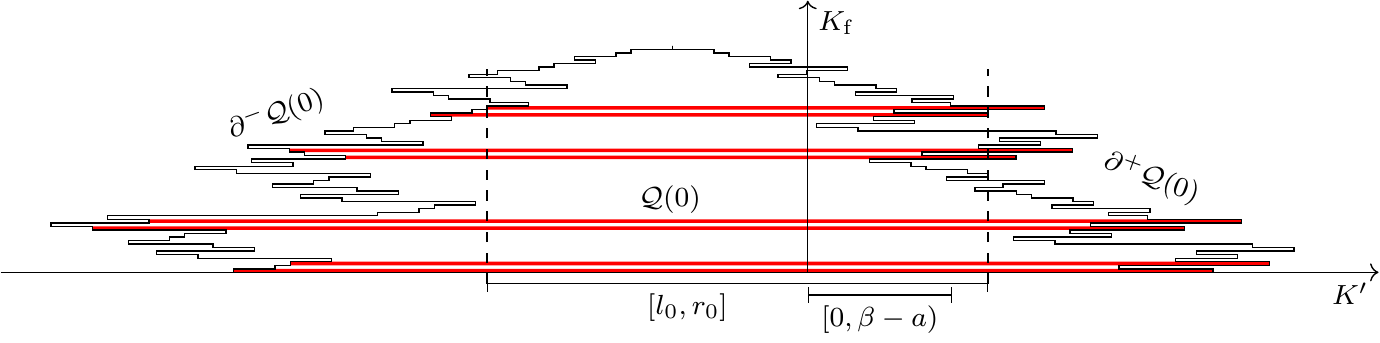}%
\cr\hfil
\TIKZinclude{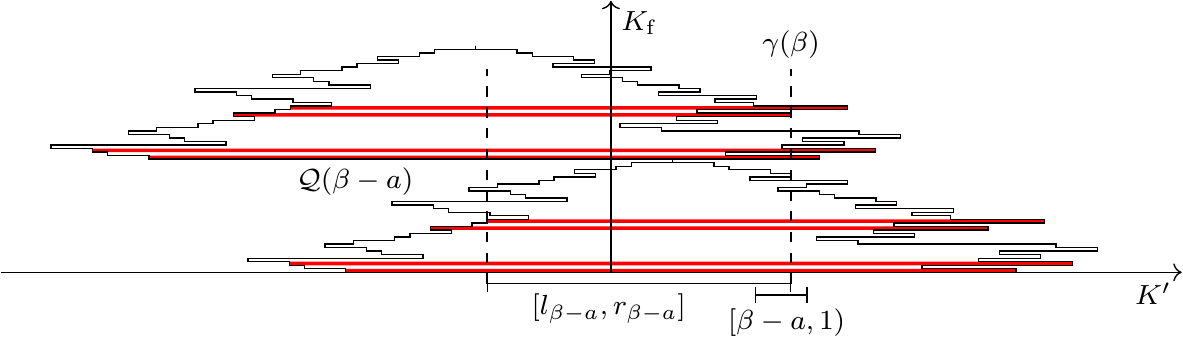}%
}
\caption{The tiles $\QQ(0)$ and $\QQ(\beta-a)$ for $\beta=1+\sqrt3$.
 The (red) stripes illustrate the intersection of $Y=K'\times (\Z)_\f$ with the tiles.}
\label{fig:intervals}
\end{figure}

We recall the results of \cite[\S9.3]{minervino_steiner_2014}, where the shape of the tiles is described.
The intersection of $\QQ(x)$ with a line $K'\times\{z\}$ is a line segment for any $z\in\ol{\Z[\beta]_\f}$ and it is empty
 for all $z\in\Kf\setminus \ol{\Z[\beta]_\f}$, see Figure~\ref{fig:intervals}.
Let $\partial^-\QQ(x)$ denote the set of the segments' left end-points, and similarly $\partial^+\QQ(x)$
 the set of the right end-points. 
Put
\[
	l_x\eqdef\sup\pi'(\delta^-\QQ(x)\cap Y)
\quad\text{and}\quad
	r_x\eqdef\inf\pi'(\delta^+\QQ(x)\cap Y)
\quad\text{for $x=0,\beta-a$}
,\]
 where $Y\eqdef K'\times(\Z_b)_\f$ and $\pi'$ denotes the projection $\pi'\colon K'\times \Kf\to K'$, $(y,z)\mapsto y$.
Then all numbers $p/q\in\Z_b$ in $[l_0,r_0]\cap[0,\beta-a)$ have a purely periodic expansion,
 and so do all numbers $p/q\in\Z_b$ in $[l_{\beta-a},r_{\beta-a}]\cap[\beta-a,1)$.
Outside these two sets, numbers $p/q\in\Z_b$ that do not have a purely periodic expansion are dense,
 since the points $\delta'(p/q)$ are dense in $Y$ by Lemma~\ref{lem:closure}.
Therefore, the value of $\gamma(\beta)$ depends on the relative position of the above intervals
 (see Figure~\ref{fig:intervals}) in the following way:
\begin{equation}\label{eq:gamma-Y}
	\gamma(\beta) = \begin{cases}
		0
		&
		\text{if $l_0>0$ or $r_0<0$,}
	\\
		r_0
		&
		\text{if $l_0\leq0$ and $r_0\in[0,\beta-a)$,}
	\\
		\beta-a
		&
		\text{if $l_0\leq0$, $r_0\geq\beta-a$ and $\beta-a\notin[l_{\beta-a},r_{\beta-a}]$,}
	\\
		\min\{r_{\beta-a},1\}
		&
		\text{if $l_0\leq0$, $r_0\geq\beta-a$ and $\beta-a\in[l_{\beta-a},r_{\beta-a}]$.}
	\end{cases}
\end{equation}

In the rest of the proof, we will show that
\begin{equation}\label{eq:lr}
	l_0 = l_{\beta-a}-1 = - \beta + \sup_{j\in\Z} P_{\h(j-\beta)}(\beta')
\quad\text{and}\quad
	r_0 = r_{\beta-a} = 1 + \inf_{j\in\Z} P_{\h(j)}(\beta')
.\end{equation}
As $\inf_{j\in\Z} P_{\h(j)}(\beta')\leq P_{\h(0)}(\beta')=0$,
 we see that \eqref{eq:gamma-Y} implies the statement of the theorem.

\begin{figure}
\centering
\TIKZinclude{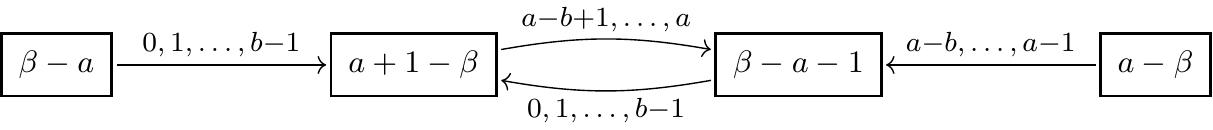}
\caption{Boundary graph for quadratic $\beta$-tiles, cf.~{\normalfont\protect\cite[Fig.~6]{minervino_steiner_2014}}.
 Each arrow in the graph represents exactly $b$~edges.}
\label{fig:graph}
\end{figure}

We use results of~\cite[\S\S8.3, 9.2 and 9.3]{minervino_steiner_2014}, namely Equations~(8.4) and~(9.2), which read:
\[
	z\in\RR(x)\cap\RR(y)
\quad\text{if and only if}\quad
	z = \delta'(x) + P_{\bm u}(\delta'(\beta))
,\]
 where $\bm u = v_0v_1v_2\dotsm$ is an edge-labelling of a path in the boundary graph in Figure~\ref{fig:graph}
 that starts in the node $y-x$;
 and
\[
	\partial\RR(x) =
		\Bigl( \RR(x) \cap \RR\bigl(x+\beta-\qfl{x+\beta}\bigr)\Bigr)
	\cup
		\Bigl( \RR(x) \cap \RR\bigl(x-\beta-\qfl{x-\beta}\bigr)\Bigr)
,\]
 where the first part is the left boundary $\RR^-(x)$ and the second part is the right boundary $\RR^+(x)$.
Therefore
\begin{gather*}
	\partial^-\RR(0)
	= \partial^+\RR(\beta-a)
	= \RR(0) \cap \RR(\beta-a)
	 = \set[\big]{ P_{\bm u}(\delta'(\beta)) }{ \bm u\in (\AA\BB)^\omega }
,\\
	\partial^+\RR(0)
	= \RR(a+1-\beta) \cap \RR(0)
	 = \set[\big]{ \delta'(a+1-\beta)+P_{\bm u}(\delta'(\beta)) }{ \bm u\in (\AA\BB)^\omega }
,\\
	\partial^-\RR(\beta-a)
	= \RR(\beta-a) \cap \RR(2\beta-\qfl{2\beta})
	 = \set[\big]{ \delta'(\beta-a) + P_{\bm u}(\delta'(\beta)) }{ \bm u\in(\AA\BB)^\omega }
,\end{gather*}
 where we put $\BB\eqdef\{a{-}b{+}1,a{-}b{+}2,\dots,a\}$.
We have that
\begin{align*}
	\set[\big]{P_{\bm u}(\delta'(\beta))}{\bm u\in(\AA\BB)^\omega}
	&= \set[\big]{P_{\per{(b-1)a}}(\delta'(\beta)) - P_{\bm u}(\delta'(\beta))}{\bm u\in\AA^\omega}
\\
	&= -\delta'(1)-\set[\big]{P_{\bm u}(\delta'(\beta))}{\bm u\in\AA^\omega}
,\end{align*}
 since $\AA=b-1-\AA$ and $\BB=a-\AA$.
Because $\QQ(x)=\delta'(x)-\RR(x)$, we have $\partial^{\pm}\QQ(x) = \delta'(x)-\partial^{\mp}\RR(x)$.
We obtain
\begin{gather*}
	\partial^-\QQ(0)
	= \delta'(\beta-a)+\set[\big]{P_{\bm u}(\delta'(\beta))}{\bm u\in\AA^\omega}
,\\
	\partial^-\QQ(\beta-a)
	= \delta'(\beta-a+1)+\set[\big]{P_{\bm u}(\delta'(\beta))}{\bm u\in\AA^\omega}
,\\
	\partial^+\QQ(0) = \partial^+\QQ(\beta-a)
	= \delta'(1)+\set[\big]{P_{\bm u}(\delta'(\beta))}{\bm u\in\AA^\omega}
.\end{gather*}

We have that
\[
	\delta'(1)+P_{\bm u}(\delta'(\beta))\in Y
\iff
	1_\f + P_{\bm u}(\beta_\f) \in \Z_\f
\iff
	P_{\bm u}(\beta_\f)\in \Z_\f
\iff
	\bm u\in\h_\f(\Z_\f) 
,\]
 because $\h_\f(P_{\bm u}(\beta_\f))=\bm u$ and $\h_\f$ is a homeomorphism by Lemma~\ref{lem:H}.
Then, since the map $\Z_\f\to K'$, $z\mapsto P_{\h_\f(z)}(\beta')$ is continuous, we get that
\[
	\inf \pi'\bigl(\partial^+\QQ(x)\cap Y\bigr)
	= 1 + \inf_{z\in\Z_\f} P_{\h_\f(z)}(\beta')
	= 1 + \inf_{j\in\Z} P_{\h(j)}(\beta')
.\]
Similarly, $\delta'(\beta-a)+P_{\bm u}(\delta'(\beta)) \in Y$ if and only if $\bm u\in\h_\f(\Z_\f-\beta_\f)$, therefore
\[
	\sup \pi'\bigl(\partial^-\QQ(\beta-a)\cap Y\bigr)-1
	= \sup \pi'\bigl(\partial^-\QQ(0)\cap Y\bigr)
	= \beta'-a+\sup_{j\in\Z} P_{\h(j-\beta)}(\beta')
.\]
Since $\beta'-a=-\beta$, this justifies~\eqref{eq:lr}.
\end{proof}

\begin{proof}[of Theorem~\ref{thm:gamma0}, case $a>\frac{1+\sqrt5}{2}b$]
Since $\beta'<0$, we have that 
\[
	\sup_{j\in\Z} P_{\h(j-\beta)}(\beta')
	\leq \sup_{\bm u\in\AA^\omega} P_{\bm u}(\beta')
	= P_{\per{(b{-}1)0}}(\beta')
	= \frac{b-1}{1-(\beta')^2}
.\]
We will show that this quantity is $<2\beta-a-1$.
First, we derive, using $(\beta')^2=a\beta'+b$, $\beta=a-\beta'$ and $1-(\beta')^2>0$, that it is equivalent to
\begin{equation}\label{eq:ffff}
	a+ab+\beta'(a^2+a+2b-2) > 0
.\end{equation}
We know that $\beta<a+1$, therefore $\beta=a+\frac{b}{\beta}>\frac{a(a+1)+b}{a+1}$ and
 $\beta'=-\frac{b}{\beta}>-\frac{(a+1)b}{a^2+a+b}$.
As well, $a^2+a+2b-2>0$, therefore we estimate
\[
	a+ab+\beta'(a^2+a+2b-2) > \frac{ab^2((\frac ab)^2-\frac ab-1)+b^2((\frac ab)^2+2\frac ab-2)+2b}{a^2+a+b}
.\]
When $\frac ab>\frac{1+\sqrt5}{2}$, all three terms in the numerator are positive.
Since the denominator is also positive, we get that $\sup_{j\in\Z} P_{\h(j-\beta)}(\beta') < 2\beta-a-1$.
Theorem~\ref{thm:gamma-bar} then implies~\eqref{eq:gamma0}.
\end{proof}

The proof of the case $a\perp b$ of Theorem~\ref{thm:gamma0} was given in~\cite[\S9]{minervino_steiner_2014}.
The proof of the case $a=b$ is given in the next section on page~\pageref{pf:gamma0:a=b},
 because it falls under the case when $b$ divides $a$.

\medskip

The following proposition shows how to compute the infimum in Theorem~\ref{thm:gamma0}
 and thus the value of $\gamma(\beta)$ in a lot of (and possibly all) cases.
Comments on the computation of $\gamma(\beta)$ by Theorem~\ref{thm:gamma-bar} are in Section~\ref{sect:general}.
We recall that $\Pref[n]{\bm u}$ denotes the prefix of $\bm u$ of length $n$.

\begin{proposition}\label{prop:estim}
Let $\beta^2=a\beta+b$ with $a\geq b\geq2$.
Then for each $n\in\N$ we have
\begin{equation}\label{eq:estim}
	\inf_{j\in\Z} P_{\h(j)}(\beta')
	\in \min_{\!j\in\{0,1,\dots,b^n-1\}\!} P_{\Pref[n]{\h(j)}}(\beta')
	+ (\beta')^n\tfrac{b-1}{1-(\beta')^2}[\beta',1]
.\end{equation}
\end{proposition}

\begin{lemma}\label{lem:pref-n}
Let $x,y\in\Z[\beta]$ satisfy that $x-y\in b^n\Z[\beta]$.
Then $\Pref[n]{\h(x)} = \Pref[n]{\h(y)}$.
\end{lemma}

\begin{proof}
Since $b=\beta^2-a\beta\in\beta\Z[\beta]$, we have that $x-y\in\beta^n\Z[\beta]$.
Let $\h(x)=u_0u_1\dotsm $.
Then $x-\sum_{j=0}^{n-1} u_j\beta^j\in\beta^n\Z[\beta]$ and therefore $y-\sum_{j=0}^{n-1} u_j\beta^j\in\beta^n\Z[\beta]$,
 which means that $u_0\dotsm u_{n-1}$ is a~prefix of $\h(y)$.
\end{proof}

\begin{proof}[of Proposition~\ref{prop:estim}]
Set $\mu_n\eqdef \min_{j\in\{0,1,\dots,b^n-1\}} P_{\Pref[n]{\h(j)}}(\beta')$.
The statement actually consists of two inequalities, which will be proved separately.
Let $j\in\Z$.
Since $\Pref[n]{\h(j)}=\Pref[n]{\h(j\omod b^n)}$ by Lemma~\ref{lem:pref-n}
 and since $\beta'<0$, we have
\begin{alignat*}{2}&
	P_{\h(j)}(\beta')
	\geq	P_{\Pref[n]{\h(j)}\per{0(b{-}1)}}(\beta')
	\geq \mu_n + (\beta')^{n+1}\tfrac{b-1}{1-(\beta')^2}
&\quad&
	\text{if $n$ is even,}
\\&
	P_{\h(j)}(\beta')
	\geq	P_{\Pref[n]{\h(j)}\per{(b{-}1)0}}(\beta')
	\geq \mu_n + (\beta')^{n}\tfrac{b-1}{1-(\beta')^2}
&&
	\text{if $n$ is odd.}
\intertext{
To prove the other inequality, let $k\in\{0,\dots,b^n-1\}$ be such that $\mu_n=P_{\Pref[n]{\h(k)}}(\beta')$.
Then
}&
	P_{\h(k)}(\beta')
	\leq P_{\Pref[n]{\h(k)}\per{(b{-}1)0}}(\beta')
	= \mu_n + (\beta')^{n}\tfrac{b-1}{1-(\beta')^2}
&&
	\text{if $n$ is even,}
\\&
	P_{\h(k)}(\beta')
	\leq P_{\Pref[n]{\h(k)}\per{0(b{-}1)}}(\beta') = \mu_n + (\beta')^{n+1}\tfrac{b-1}{1-(\beta')^2}
&&
	\text{if $n$ is odd;}
\end{alignat*}
 this provides the upper bound on the infimum. 
\end{proof}

\section{The case \texorpdfstring{$b$}{b} divides \texorpdfstring{$a$}{a}}\label{sect:b-div-a}

In this section, we aim to prove Theorem~\ref{thm:abZ},
 which deals with the particular case when $b$ divides~$a$.
Table~\ref{tab:star} shows whether $\gamma(\beta)$ is $0$, $1$ or strictly in between,
 for $b\leq12$ and $a/b\leq15$.
The first non-trivial values are listed in Table~\ref{tab:num}.
The algorithm for obtaining these values is deduced from Theorem~\ref{thm:gamma0}
 (which covers all the cases when $\frac ab\in\Z$
 since then either $a=b$ or $a\geq 2b>\frac{1+\sqrt5}{2}b$), and the following proposition,
 which improves the statement of Proposition~\ref{prop:estim}.

\begin{table}[t]
\centering\normalsize
\newcommand*\x[1]{\multicolumn{1}{@{}>$c<$@{}}{#1}}
\begin{tabular}{>$r<$|*{15}{@{}>{\hbox to 14.4pt\bgroup\hfil\rule[-3.8pt]{0pt}{14.4pt}$}c<{$\hfil\egroup}@{}|}>$r<$}
\multicolumn{1}{c}{} & \x{\mathllap{a/b={}}1} & \x2 & \x3 & \x4 & \x5 & \x6 & \x7 & \x8 & \x9 & \x{10} & \x{11} & \x{12} & \x{13} & \x{14} & \x{15} & \phantom{b=1}\\\cline{2-16}
b=1& 1 & 1 & 1 & 1 & 1 & 1 & 1 & 1 & 1 & 1 & 1 & 1 & 1 & 1 & 1 \\\cline{2-16}
 2 & \s& 1 & 1 & 1 & 1 & 1 & 1 & 1 & 1 & 1 & 1 & 1 & 1 & 1 & 1 \\\cline{2-16}
 3 & 0 & \s& 1 & 1 & 1 & 1 & 1 & 1 & 1 & 1 & 1 & 1 & 1 & 1 & 1 \\\cline{2-16}
 4 & 0 & \s& \s& 1 & 1 & 1 & 1 & 1 & 1 & 1 & 1 & 1 & 1 & 1 & 1 \\\cline{2-16}
 5 & 0 & \s& \s& \s& 1 & 1 & 1 & 1 & 1 & 1 & 1 & 1 & 1 & 1 & 1 \\\cline{2-16}
 6 & 0 & \s& \s& 1 & 1 & 1 & 1 & 1 & 1 & 1 & 1 & 1 & 1 & 1 & 1 \\\cline{2-16}
 7 & 0 & \s& \s& \s& \s& \s& 1 & 1 & 1 & 1 & 1 & 1 & 1 & 1 & 1 \\\cline{2-16}
 8 & 0 & \s& \s& \s& \s& \s& \s& 1 & 1 & 1 & 1 & 1 & 1 & 1 & 1 \\\cline{2-16}
 9 & 0 & \s& \s& \s& \s& \s& \s& \s& 1 & 1 & 1 & 1 & 1 & 1 & 1 \\\cline{2-16}
10 & 0 & \s& \s& \s& \s& \s& \s& \s& \s& 1 & 1 & 1 & 1 & 1 & 1 \\\cline{2-16}
11 & 0 & 0 & \s& \s& \s& \s& \s& \s& \s& \s& 1 & 1 & 1 & 1 & 1 \\\cline{2-16}
12 & 0 & 0 & \s& \s& \s& \s& \s& \s& \s& \s& \s& 1 & 1 & 1 & 1 \\\cline{2-16}
\end{tabular}\\*[1ex]
\caption{The values of $\gamma(\beta)$ for the case when $b$ divides $a$. The star `$\s$' means that the value is strictly between $0$ and $1$.}
\label{tab:star}
\end{table}

\begin{proposition}\label{prop:estim-2}
Let $\beta^2=a\beta+b$ with $a\geq b\geq2$ and $\frac ab\in\Z$.
Then for each $n\in\N$ we have
\begin{equation*}
	\inf_{j\in\Z} P_{\h(j)}(\beta')
	\in \min_{j\in\{0,1,\dots,b^n-1\}} P_{\Pref[2n]{\h(j)}}(\beta')
	+ (\beta')^{2n}\tfrac{b-1}{1-(\beta')^2}[\beta',0]
.\end{equation*}
\end{proposition}

\begin{lemma}\label{lem:pref-n-2}
Let $\beta^2=cb\beta+b$.
Let $x,y\in\Z[\beta]$ satisfy that $x-y\in b^n\Z[\beta]$ for some $n\in\N$.
Then $\Pref[2n]{\h(x)}=\Pref[2n]{\h(y)}$.
Moreover, for all $x\in\Z[\beta]$ and $d\in\AA$ there exists $y\in x+b^n\AA$ such that
 $\Pref[2n\tight+1]{\h(y)}=\Pref[2n]{\h(x)}d$.
\end{lemma}

\begin{proof}
We have $\beta^2=b(c\beta+1)\in b\Z[\beta]$ and $b=\beta^2-c(1+c^2b)\beta^3+c^2\beta^4\in\beta^2+\beta^3\Z[\beta]\subseteq\beta^2\Z[\beta]$,
 whence $\beta^2\Z[\beta]=b\Z[\beta]$ and $\beta^{2n}\Z[\beta]=b^n\Z[\beta]$ for all $n\in\N$.
Following the lines of the proof of Lemma~\ref{lem:pref-n},
 we obtain that if $x-y\in b^n\Z[\beta]$ then $\h(x)$ and $\h(y)$ have a common prefix of length at least $2n$.

Put $\h(x)=u_0u_1\dotsm$.
Since $b^n\in\beta^{2n}+\beta^{2n+1}\Z[\beta]$,
 we have that $u_0u_1\dotsm u_{2n-1}d$ is a prefix of $\h(x+eb^n)$ for any $e\equiv d-u_{2n}\pmod b$.
\end{proof}

\begin{proof}[of Proposition~\ref{prop:estim-2}]
We follow the lines of the proof of Proposition~\ref{prop:estim} for the case $n$ even.
The lower bound is the same in both statements, therefore we only need to prove that
 $\inf_{j\in\Z} P_{\h(j)}(\beta')\leq P_{\Pref[2n]{\h(k)}}(\beta')$, 
 where $k\eqdef\arg\min_{j\in\{0,\dots,b^n-1\}} P_{\Pref[2n]{\h(j)}}(\beta')$.
For each $m\in\N$, there exists $k_m\in\Z$ such that $\Pref[2n\tight+2m]{\h(k_m)} \in \Pref[2n]{\h(k)}(0\AA)^m$
 by Lemma~\ref{lem:pref-n-2}.
Then
\[
	\inf_{j\in\Z} P_{\h(j)}(\beta')
	\leq \inf_{m\in\N} P_{\h(k_m)}(\beta')
	\leq \inf_{m\in\N} P_{\Pref[n]{\h(k)}(00)^m\per{(b{-}1)0}}(\beta')
	= P_{\Pref[n]{\h(k)}}(\beta')
.\qedhere\]
\end{proof}

\begin{remark}\label{rem:J-2}
We have that
\begin{equation}\label{eq:mu-2}
	\mu_n \eqdef
	\min_{j\in\{0,1,\dots,b^n-1\}} P_{\Pref[2n]{\h(j)}}(\beta')
	= \min_{j\in J_{n-1}+b^{n-1}\AA} P_{\Pref[2n]{\h(j)}}(\beta')
,\end{equation}
 where
\[
	J_0\eqdef\{0\}
\quad\text{and}\quad
	J_n\eqdef \set[\big]{j\in J_{n-1}+b^{n-1}\AA}
		{P_{\Pref[2n]{\h(j)}}(\beta')<\mu_n+\abs{\beta'}^{2n+1}\tfrac{b-1}{1-(\beta')^2}}
.\]
To verify~\eqref{eq:mu-2}, we first show that the sequence $(\mu_n)_{n\in\N}$ is non-increasing.
Let $j\in\{0,\dots,b^n-1\}$ be such that $\mu_n=P_{\Pref[2n]\h(j)}(\beta')$.
Then by Lemma~\ref{lem:pref-n-2} there exists $d\in\AA$ such that $\Pref[2n\tight+1]{\h(j+db^n)}=\Pref[2n]{\h(j)}\mkern1mu0$,
 whence $\mu_{n+1}\leq P_{\Pref[2n+2]{\h(j+db^n)}}(\beta')\leq\mu_n$.

Suppose now that $j\in \{0,\dots,b^n-1\}\setminus(J_{n-1}+b^{n-1}\AA)$.
Then there exists $m<n$ such that $P_{\Pref[2m]{\h(j)}}(\beta')\geq \mu_m + \abs{\beta'}^{2m+1}\tfrac{b-1}{1-(\beta')^2}$,
 therefore $P_{\Pref[2n]{\h(j)}}(\beta')> \mu_m\geq \mu_n$.
\end{remark}

\begin{example}\label{ex:22}
As an example, the computation of $\gamma(\beta)$ for $\beta=1+\sqrt3$, the Pisot root of $\beta^2=2\beta+2$,
 is visualized in Figure~\ref{fig:comp}.
For each step of the algorithm, the value of $\gamma(\beta)$ lies in the left-most interval.
Already in the 5th step we obtain that $\gamma(\beta)\in[0.900834,0.970552]$, therefore it is strictly between $0$ and $1$.
Note that in the 9th step we have that $\mu_9=P_{t^{(9)}}(\beta')$ with $t^{(9)}=001100010101010001$,
 and $\gamma(\beta)\in[0.91012665225, 0.91587668314]$.
In the 40th step, we have that $t^{(40)}=
 00 11 00 (01)^4 00 01 00 (00 01)^4 (00)^2 (01)^5 (00)^3 (01)^6 (00)^2 01$
 and $\gamma(\beta)\approx0.914803044$. In the 200th step, we obtain
\[
\iffiif
\else 
	\gamma(1\tight+\sqrt3)\approx0.
	9148030441966581950472931393
	\text{\small$9379415269499861873397617573$}
	\text{\footnotesize$3141835762361$}
\fi 
.\]
\end{example}

\begin{table}[t]
\centering\normalsize
\begin{tabular}[t]{*{2}{>$c<$}>$l<$}
\toprule
a & b & $\hfill$\gamma(\beta)\phantom{00\cdots}$\hfill$
\\\midrule[\heavyrulewidth]
 2 & 2 & 0.914803044196\cdots \\\midrule
 6 & 3 & 0.992963560101\cdots \\\midrule
 8 & 4 & 0.933542944675\cdots \\
12 & 4 & 0.999897789000\cdots \\\midrule
10 & 5 & 0.834150794175\cdots \\
15 & 5 & 0.995306723671\cdots \\
20 & 5 & 0.999999907110\cdots\phantom{00} \\
\bottomrule
\end{tabular}\qquad
\begin{tabular}[t]{*{2}{>$c<$}>$l<$}
\toprule
a & b & $\hfill$\gamma(\beta)\phantom{00\cdots}$\hfill$
\\\midrule[\heavyrulewidth]
12 & 6 & 0.736114178272\cdots \\
18 & 6 & 0.993897266395\cdots \\\midrule
14 & 7 & 0.584906533458\cdots \\
21 & 7 & 0.944526094618\cdots \\
28 & 7 & 0.997984788082\cdots \\
35 & 7 & 0.999986041767\cdots \\
42 & 7 & 0.99999999999971\cdots \\
\bottomrule
\end{tabular}
\iffiif\else\vspace*{1ex}\fi 
\caption{Numerical values of $\gamma(\beta)$, where $\beta^2=a\beta+b$, that correspond to the first~couple~`$\s$'~in~Table~\ref{tab:star}.}
\label{tab:num}
\end{table}

\begin{figure}[t]
\centering
\TIKZinclude{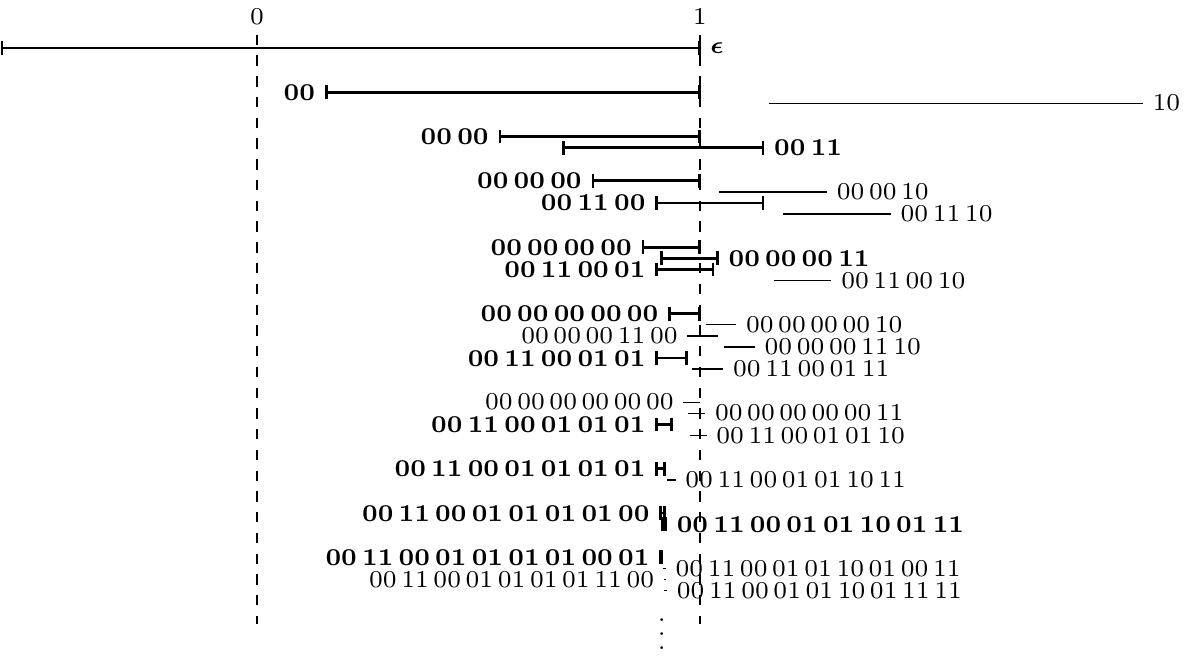}%
\caption{The computation of $\gamma(1+\sqrt3)$.
 By a~thick line with a bold label we denote the intervals that we `keep' (these arise from numbers in $J_n$),
  by a~thin line the ones that we `forget'.
 The labels next to the intervals are the corresponding prefixes $\Pref[2n]{\h(j)}$.}
\label{fig:comp}
\end{figure}

\begin{proof}[of Theorem~\ref{thm:gamma0}, case $a=b$]\phantomsection\label{pf:gamma0:a=b}
Take $a=b\geq4$.
Then $b=\beta^2+(b-1)\beta^3+(2b+1)\beta^4$, therefore $\Pref[4]{\h(b)}=001(b{-}1)$.
According to Proposition~\ref{prop:estim-2}, we have that
\[
	A
	\eqdef \inf_{j\in\Z} P_{\h(j)}(\beta')
	\leq P_{001(b{-}1)}(\beta')
	= (\beta')^2+(b-1)(\beta')^3
.\]
For $a=b\geq5$, we use the estimate $-\beta'\in(\frac{b}{b+1},1)$ to obtain that $A<1-\frac{b^3(b-1)}{(b+1)^3}<-1$,
 therefore $\gamma(\beta)=0$.
For $a=b=4$, we have $P_{001(b{-}1)}(\beta')\approx -1.0193$, thus $A<-1$.

When $a=b=3$, we verify that $\Pref[12]{\h(21)}=001200020201$ and Proposition~\ref{prop:estim-2} yields
\(
	A\leq P_{001200020201}(\beta') \approx -1.0726<-1
,\)
 therefore $\gamma(\beta)=0$.

When $a=b=2$, we can follow the lines of the proof of the case $a>\frac{1+\sqrt5}{2}b$,
 because we observe that \eqref{eq:ffff} is satisfied, namely $6+8\beta'\approx0.1436>0$.
\end{proof}

The proof of Theorem~\ref{thm:abZ} is divided into several cases.

\begin{proof}[of Theorem~\ref{thm:abZ}, case $a\geq b^2$]
Any $j\in\Z\setminus\{0\}$ can be written as $j=b^n(j_0+j_1b)$, where $n\in\N$, $j_0\in\AA\setminus\{0\}$ and $j_1\in\Z$.
Then $\Pref[2n\tight+1]{\h(j)}=0^{2n}j_0$ becase $b^n\in\beta^{2n}+\beta^{2n+1}\Z[\beta]$,
 whence
\begin{multline*}
	P_{\h(j)}(\beta')
	\geq P_{\Pref[2n\tight+1]{\h(j)}((b-1)0)^\omega}(\beta')
	\geq P_{0^{2n}1((b-1)0)^\omega}(\beta')
\\[0.5ex]
	= (\beta')^{2n}\Bigl( 1 + \frac{(b-1)\beta'}{1-(\beta')^2} \Bigr)
	= (\beta')^{2n}\Bigl( 1 - \frac{(b-1)b\beta}{\beta^2-b^2} \Bigr)
	> 0
,\end{multline*}
 where the last inequality was already proved in \cite[Theorem~6]{minervino_steiner_2014}.
As $\h(0)=0^\omega$, we have $P_{\h(0)}(\beta')=0$.
From Theorem~\ref{thm:gamma0} we conclude that $\gamma(\beta) = 1 + \inf_{j\in\Z} P_{\h(j)}(\beta') = 1$.
\end{proof}

The remaining cases of the proof of Theorem~\ref{thm:abZ} make use of the following relations.
Let $c\eqdef a/b\in\Z$.
Then $\frac{b}{\beta^2} = \frac{1}{1+c\beta} \in 1-c\beta+c^2\beta^2-c^3\beta^3+\beta^4\Z[\beta]$,
 and more generally,
\begin{equation}\label{eq:binom}
	\frac{b^n}{\beta^{2n}} 
	\in 1-nc\beta+\tbinom{n+1}{2} c^2\beta^2-\tbinom{n+2}{3}c^3\beta^3 + \beta^4\Z[\beta]
	\quad\text{for any $n\in\N$.}
\end{equation}
For $j=(j_0+j_1b)b^n$ with $n\in\N$, and $j_0,j_1\in\Z$ we have that $\frac{j}{\beta^{2n}}=j_0\frac{b^n}{\beta^{2n}}+j_1\beta^2\frac{b^{n+1}}{\beta^{2n+2}}$, therefore
\begin{equation}\label{eq:j-beta2m}
	\frac{j}{\beta^{2n}} \in j_0-j_0nc\beta+\Bigl(j_0\tbinom{n+1}{2}c^2+j_1\Bigr)\beta^2
	 - \Bigl(j_0\tbinom{n+2}{3}c^3+j_1(n+1)c\Bigr)\beta^3+\beta^4\Z[\beta]
.\end{equation}

\begin{proof}[of Theorem~\ref{thm:abZ}, case $\beta^2=30\beta+6$]
We have $b=6$ and $c=5$.
As in the proof of the previous case, we will show that $P_{\h(j)}(\beta')\geq0$ for all $j\in\Z$.
Let $j\neq0$ be written as $j=b^n(j_0+j_1b)$ with $j_0\in\AA\setminus\{0\}$ and $j_1\in\Z$,
 then $\h(j)=0^{2n}u_0u_1u_2\dotsm$ for some $u_0u_1\dotsm\in\AA^\omega$ with $u_0=j_0$,
 and $P_{\h(j)}(\beta') = (\beta')^{2n}P_{u_0u_1\dotsm}(\beta')$.
We consider the following cases:
\begin{itemize}
\item If $u_0\geq2$, then $P_{u_0u_1\dotsm}(\beta')\geq P_{2\per{50}}(\beta')>0$.
\item If $u_0=1$ and $u_1\leq4$, then $P_{u_0u_1\dotsm}(\beta')\geq P_{14\per{05}}(\beta')>0$.
\item If $u_0u_1=15$,
 then~\eqref{eq:j-beta2m} yields that $j_0=1$ and $-j_0nc\equiv 5\pmod 6$, therefore $n\equiv-1\pmod 6$ and $n=6n_1-1$, i.e.,
 $-j_0nc\beta=5\beta-30n_1\beta\in 5\beta-5n_1\beta^3+\beta^4\Z[\beta]$.
Therefore
\[
	\frac{j}{\beta^{2n}} \in 1+5\beta + \Bigl(\tbinom{6n_1}{2}5^2+j_1\Bigr)\beta^2
	 - \Bigl(\tfrac{(6n_1+1)6n_1(6n_1-1)}{6}5^3+30n_1j_1+5n_1\Bigr)\beta^3
	 + \beta^4\Z[\beta]
.\]
The coefficient of $\beta^3$ is congruent to $0$ modulo $6$ regardless of the values of $n_1$ and $j_1$.
This means that $u_3=0$.
Then $P_{15u_20\per{05}}(\beta')\geq P_{1500\per{05}}(\beta')>0$.
\end{itemize}
Therefore we have $P_{\h(j)}(\beta')\geq0$ for all $j\in\Z$.
\end{proof}

\begin{proof}[of Theorem~\ref{thm:abZ}, case $\beta^2=24\beta+6$]
We have $b=6$ and $c=4$.
We use the same technique as in the case $\beta^2=30\beta+6$.
\begin{itemize}
\item If $u_0\geq2$, then $P_{u_0u_1\cdots}(\beta')\geq P_{2\per{50}}(\beta')>0$.
\item If $u_0=1$ and $u_1\leq3$, then $P_{u_0u_1\cdots}(\beta')\geq P_{13\per{05}}(\beta')>0$.
\item Since $c$ is even, we get that $u_1\equiv -j_0nc\pmod{6}$ is even, therefore $u_0u_1\neq15$.
\item If $u_0u_1=14$, then \eqref{eq:j-beta2m} gives $j_0=1$ and $-j_0nc\equiv 4\pmod6$, i.e., $n\equiv -1\pmod3$ and $n=3n_1-1$,
 whence $-j_0nc\beta=4\beta-12n_1\beta\in 4\beta-2n_1\beta^3+\beta^4\Z[\beta]$.
We derive that
\[
	\frac{j}{\beta^{2n}} \in 1+4\beta+(\text{some\ integer})\beta^2-\bigl(144n_1^3-30n_1+12n_1j_1\bigr)\beta^3+\beta^4\Z[\beta]
.\]
As above, we get that $u_3=0$ regardless of the values of $n_1$ and $j_1$,
thus $P_{u_0u_1\cdots}(\beta') \geq P_{1400\per{05}}(\beta')>0$.
\qedhere
\end{itemize}
\end{proof}

\begin{proof}[of Theorem~\ref{thm:abZ}, case $c\eqdef a/b<b$ and $c\notin\{4,5\}$ when $b=6$]
Let $n \eqdef \lceil \frac{c}{b-c} \rceil$.
From \eqref{eq:binom}, the $\beta$-adic expansion $\h(b^n)$ starts with $0^{2n}1(nb{-}nc)$.
If $\frac{c}{b-c} \notin \mathbb{Z}$, then we have $nb-nc > c$ and thus $P_{1(nb{-}nc)}(\beta') \leq 1 + (c+1) \beta' < 0$,
 using that $\beta' = -\frac{b}{\beta} < -\frac{b}{cb+1} \leq -\frac{1}{c+1}$. 
By Proposition~\ref{prop:estim-2}, this proves that $\gamma(\beta) < 1$ if $c$ is not a~multiple of $b-c$.

Assume now that $\frac{c}{b-c} \in \mathbb{Z}$, i.e., $n = \frac{c}{b-c}$.
For $j\eqdef b^n-\tbinom{n+1}{2}c^2b^{n+1}$, we have by~\eqref{eq:j-beta2m} that
\[
	\frac{j}{\beta^{2n}} \in 1-nc\beta-\Bigl( \tbinom{n+2}{3} c^3-\tbinom{n+1}{2}c^3(n+1)\Bigr)\beta^3 + \beta^4 \mathbb{Z}[\beta]
.\]
Since $-nc=c-nb\in c-n\beta^2+\beta^3\Z[\beta]$ and $(n+1)c=nb\in\beta\Z[\beta]$, we obtain that
\[
	\frac{j}{\beta^{2n}} \in 1+c\beta-\Bigl( \tbinom{n+2}{3} c^3+n\Bigr)\beta^3 + \beta^4 \mathbb{Z}[\beta]
.\]
If $\tbinom{n+2}{3} c^3+n\not\equiv0\pmod b$, then
\[
	P_{\Pref[2n\tight+4]{\h(j)}}(\beta') 
	\leq P_{0^{2n}1c01}(\beta')
	= \frac{(\beta')^{2n+2}}{b} + (\beta')^{2n+3} = (\beta')^{2n+2} \frac{\beta-b^2}{b\beta} < 0
,\]
 since $1+c\beta'=\frac{(\beta')^2}b$ and $\beta < a + 1 \le b^2$,
 therefore $\gamma(\beta)<1$ by Proposition~\ref{prop:estim-2}.

It remains to consider the case that $\binom{n+2}{3} c^3 + n \equiv 0 \bmod b$, i.e.,
\[
	n \equiv - \frac{bn(n+2)}{6} c^2 n \bmod b
,\]
because $(n+1)c=nb$.
Multiplying by $b-c$ gives
\[
	c \equiv - \frac{bn(n+2)}{6} c^3 \bmod b
.\]
Note that $\frac{bn(n+2)}{6} = (b-c) \binom{n+2}{3} \in \mathbb{Z}$. 
We distinguish four cases:
\begin{enumerate}

\item
 If $6 \perp b$, then $c \equiv 0 \bmod b$, contradicting that $1 \le c < b$. 

\item
 If $2 \divides b$ and $3 \ndivides b$, then $c$ is a~multiple of~$b/2$, i.e., $c = b/2$, $n = 1$. 
 As $n$ is also a~multiple of~$b/2$, we get that $b = 2$, thus $c = 1$. 
 For $\beta^2 = 2\beta + 2$, we already know that $\gamma(\beta) < 1$, see Example~\ref{ex:22}.

\item
 If $3 \divides b$ and $2 \ndivides b$, then $c$ and $n$ are multiples of~$b/3$.
 For $c=b/3$ we have $n\notin\Z$.
 For $c = 2b/3$, we have $n = 2$,
  thus $b\in\{3,6\}$.
 However, $b=6$ contradicts $2 \ndivides b$
 and $b=3$ (i.e., $c = 2$) contradicts $\binom{n+2}{3} c^3 + n \equiv 0 \bmod b$.

\item 
 If $6 \divides b$, then $c$ and $n$ are multiples of~$b/6$,
  thus $c \in \{b/2, 2b/3, 5b/6\}$, $n \in \{1, 2, 5\}$. 
 If $n = 1$, then $b = 6$, thus $c = 3$, and $\binom{n+2}{3} c^3 + n \not\equiv 0 \bmod b$.
 If $n = 2$, then $b \in \{6, 12\}$; we have excluded that $b = 6$, $c = 4$; for $b = 12$, $c = 8$,
  we have $\binom{n+2}{3} c^3 + n \not\equiv 0 \bmod b$.
 If $n = 5$, then $b \in \{6, 30\}$; we have excluded that $b = 6$, $c = 5$; for $b = 30$, $c = 24$,
  we have $\binom{n+2}{3} c^3 + n \not\equiv 0 \bmod b$.
\qedhere

\end{enumerate}
\end{proof}

\section{The general case}\label{sect:general}

In the general quadratic case where $1<\gcd(a,b)<b$,
 the conditions of Theorem~\ref{thm:gamma0} need not be satisfied.
This means that we have to rely on the more general Theorem~\ref{thm:gamma-bar}, i.e.,
 to compute the two values $\inf_{j\in\Z} P_{\h(j)}(\beta')$ and
 $\sup_{j\in\Z} P_{\h(j-\beta)}(\beta')$.

We can derive, in a similar manner to Proposition~\ref{prop:estim}, that for all $n\in\N$,
\begin{equation}\label{eq:estim-beta}
	\sup_{j\in\Z} P_{\h(j-\beta)}(\beta')
	\in
	\max_{j\in\{0,1,\dots,b^n-1\}} P_{\Pref[n]{\h(j-\beta)}}(\beta')
	+ (\beta')^n\tfrac{b-1}{1-(\beta')^2}[\beta',1]
.\end{equation}

Let now $s_{n}\geq1$, for $n\in\N$, denote the smallest positive integer such that $s_{n}\in\beta^n\Z[\beta]$,
 and $r_n\eqdef \frac{s_n}{s_{n-1}}$.
Then $x,y\in\Z$ have a common prefix of length $n$ if and only if $y-x\in s_{n}\Z$.
Therefore, in both \eqref{eq:estim} and \eqref{eq:estim-beta} we can take $\{0,1,\dots,s_{n}-1\}$ instead of $\{0,1,\dots,b^n-1\}$.
Moreover, following Remark~\ref{rem:J-2}, we can further restrict to the sets
\begin{align*}
	J_0 & \eqdef\{0\}
,&
	J_n & \eqdef \set[\big]{ j\in J_{n-1}+s_{n-1}\{0,\dots,r_n-1\} }
		{ P_{\Pref[n]{\h(j)}}(\beta')\leq \mu_n+\abs{\beta'}^{n}\tfrac{b-1}{1+\beta'} }
,\\
	J_0' & \eqdef\{-\beta\}
,&
	J_n' & \eqdef \set[\big]{ j\in J_{n-1}+s_{n-1}\{0,\dots,r_n-1\} }
		{ P_{\Pref[n]{\h(j)}}(\beta')\geq \nu_n-\abs{\beta'}^{n}\tfrac{b-1}{1+\beta'} }
,\end{align*}
 where $\mu_n\eqdef\min_{j\in\{0,1,\dots,b^n-1\}} P_{\Pref[n]{\h(j)}}(\beta')$ and $\nu_n\eqdef\max_{j\in\{0,1,\dots,b^n-1\}} P_{\Pref[n]{\h(j-\beta)}}(\beta')$.

\bigskip

We conclude by several open questions that arise in the study of rational numbers with purely periodic expansions:
\begin{enumerate}\Alphenumi

\item
Prove or disprove that $\gamma(\beta)=1$ for quadratic Pisot number $\beta>1$, a root of $\beta^2=a\beta+b$, if and only if
 $\frac ab\in\Z$ and $a\geq b^2$ or $(a,b)\in\{(24,6),(30,6)\}$.

\item
For which quadratic $\beta$ we have that $\gamma(\beta)=0$?
Can we drop the restrictions on $a$ and $b$ in Theorem~\ref{thm:gamma0}?
More specifically, is it true that $a<\frac{1+\sqrt5}{2}b$ implies $\gamma(\beta)=0$?

\item
What is the structure of the prefixes of $\beta$-adic expansions of integers for a general quadratic~$\beta$?

\item
What about the cubic Pisot case?
Akiyama and Scheicher~\cite{akiyama_scheicher_2005} showed how to compute the value $\gamma(\beta)$ for
 $\beta\approx1.325$ the minimal Pisot number (or Plastic number), root of $\beta^3=\beta+1$.
Loridant et al.~\cite{LMST_2013} gave the contact graph of the $\beta$-tiles for cubic units,
 which could be used to determine $\gamma(\beta)$ for the units, in a similar way to what Akiyama and Scheicher did.
The consideration of the $\beta$-adic spaces could then allow the results to be expanded to non-units as well.

\end{enumerate}

\iffiif\else 
\section*{Acknowledgements}

\begingroup\footnotesize
The first author acknowledges support by Grant Agency of the Czech Technical University in Prague grant SGS11/162/OHK4/3T/14
 and Czech Science Foundation grant 13-03538S.
The second author acknowledges support by ANR/FWF project ``FAN -- Fractals and Numeration'' (ANR-12-IS01-0002, FWF grant I1136)
 and by ANR project ``Dyna3S -- Dynamique des algorithmes du pgcd'' (ANR-13-BS02-0003).

\endgroup
\fi 

\bibliographystyle{amsalpha}
\bibliography{biblio}

\providecommand{\bysame}{\leavevmode\hbox to3em{\hrulefill}\thinspace}
\providecommand{\MR}{\relax\ifhmode\unskip\space\fi MR }
\providecommand{\MRhref}[2]{%
  \href{http://www.ams.org/mathscinet-getitem?mr=#1}{#2}
}
\providecommand{\href}[2]{#2}
\begin{thebibliography}{LMST13}

\bibitem[ABBS08]{ABBS_2008}
Shigeki Akiyama, Guy Barat, Val{\'e}rie Berth{\'e}, and Anne Siegel,
  \emph{Boundary of central tiles associated with {P}isot beta-numeration and
  purely periodic expansions}, Monatsh. Math. \textbf{155} (2008), no.~3--4,
  377--419.

\bibitem[Aki98]{akiyama_1998}
Shigeki Akiyama, \emph{Pisot numbers and greedy algorithm}, Number theory
  ({E}ger, 1996), de Gruyter, Berlin, 1998, pp.~9--21.

\bibitem[AS05]{akiyama_scheicher_2005}
Shigeki Akiyama and Klaus Scheicher, \emph{Intersecting two-dimensional
  fractals with lines}, Acta Sci. Math. (Szeged) \textbf{71} (2005), no.~3--4,
  555--580.

\bibitem[BS07]{berthe_siegel_2007}
Valerie Berth{\'e} and Anne Siegel, \emph{Purely periodic {$\beta$}-expansions
  in the {P}isot non-unit case}, J. Number Theory \textbf{127} (2007), no.~2,
  153--172.

\bibitem[HI97]{hama_imahashi_1997}
M.~Hama and T.~Imahashi, \emph{Periodic {$\beta$}-expansions for certain
  classes of {P}isot numbers}, Comment. Math. Univ. St. Paul. \textbf{46}
  (1997), no.~2, 103--116.

\bibitem[IR05]{ito_rao_2005}
Shunji Ito and Hui Rao, \emph{Purely periodic {$\beta$}-expansions with {P}isot
  unit base}, Proc. Amer. Math. Soc. \textbf{133} (2005), no.~4, 953--964.

\bibitem[IS01]{ito_sano_2001}
Shunji Ito and Yuki Sano, \emph{On periodic {$\beta$}-expansions of {P}isot
  numbers and {R}auzy fractals}, Osaka J. Math. \textbf{38} (2001), no.~2,
  349--368.

\bibitem[IS02]{ito_sano_2002}
\bysame, \emph{Substitutions, atomic surfaces, and periodic beta expansions},
  Analytic number theory ({B}eijing/{K}yoto, 1999), Dev. Math., vol.~6, Kluwer
  Acad. Publ., Dordrecht, 2002, pp.~183--194.

\bibitem[LMST13]{LMST_2013}
Beno{\^\i}t Loridant, Ali Messaoudi, Paul Surer, and J{\"o}rg~M. Thuswaldner,
  \emph{Tilings induced by a class of cubic {R}auzy fractals}, Theoret. Comput.
  Sci. \textbf{477} (2013), 6--31.

\bibitem[MS14]{minervino_steiner_2014}
Milton Minervino and Wolfgang Steiner, \emph{Tilings for {P}isot beta
  numeration}, Indag. Math. (N.S.) \textbf{25} (2014), no.~4, 745--773.

\bibitem[R{\'e}n57]{renyi_1957}
Alfr{\'e}d R{\'e}nyi, \emph{Representations for real numbers and their ergodic
  properties}, Acta Math. Acad. Sci. Hungar \textbf{8} (1957), 477--493.

\bibitem[Sch80]{schmidt_1980}
Klaus Schmidt, \emph{On periodic expansions of {P}isot numbers and {S}alem
  numbers}, Bull. London Math. Soc. \textbf{12} (1980), no.~4, 269--278.

\end{thebibliography}


\end{document}